\documentclass[a4paper,12pt]{article}

\usepackage[utf8]{inputenc}
\usepackage[english]{babel}
\usepackage{enumitem}
\usepackage{mleftright}
\usepackage{fullpage}
\usepackage{graphicx}
\usepackage{subfig}
\usepackage{microtype}
\usepackage{amsmath}
\usepackage{amssymb}
\usepackage{amsthm}
\usepackage[colorlinks,allcolors=red]{hyperref}

\theoremstyle{definition}
\newtheorem{dfn}{Definition}[section]
\newtheorem{asm}[dfn]{Assumption}

\theoremstyle{plain}
\newtheorem{lem}[dfn]{Lemma}
\newtheorem{thm}[dfn]{Theorem}

\newtheorem{alg}[dfn]{Algorithm}

\newcommand{\R}{\mathbb{R}}
\newcommand{\N}{\mathbb{N}}
\newcommand{\BddMul}{w}
\newcommand{\wto}{\rightharpoonup}
\newcommand{\dual}[1]{\left\langle #1 \right\rangle}
\newcommand{\scal}[1]{\left( #1 \right)_Z}

\title{
	An Augmented Lagrangian Method for Optimization Problems in Banach Spaces%
	\thanks{This research was supported by the German Research Foundation (DFG) within the priority program ``Non-smooth and Complementarity-based Distributed Parameter Systems: Simulation and Hierarchical Optimization'' (SPP 1962) under grant numbers KA 1296/24-1 and Wa 3626/3-1.}
}

\date{October 23, 2017}

\author{
	Christian Kanzow$^{\dagger}$ \and Daniel Steck$^{\dagger}$ \and Daniel Wachsmuth%
	\thanks{University of Würzburg, Institute of Mathematics, Campus Hubland Nord,
    Emil-Fischer-Str.\ 30, 97074 Würzburg, Germany;
    \{kanzow,daniel.steck,daniel.wachsmuth\}@mathematik.uni-wuerzburg.de.}
}

\begin{document}

\maketitle

{
\small\textbf{\abstractname.}
We propose a variant of the classical augmented Lagrangian method for
constrained optimization problems in Banach spaces. Our theoretical
framework does not require any convexity or second-order assumptions and
allows the treatment of inequality constraints with infinite-dimensional
image space. Moreover, we discuss the convergence properties of our
algorithm with regard to feasibility, global optimality, and KKT conditions.
Some numerical results are given to illustrate the practical viability of
the method.
\par\addvspace{\baselineskip}
}

{
\small\textbf{Keywords.}
Constrained optimization, augmented Lagrangian method, Banach space, inequality constraints, global convergence.
\par\addvspace{\baselineskip}
}

\section{Introduction}

Let $X$, $Y$ be (real) Banach spaces and let $f:X\to\R$, $g:X\to Y$ be
given mappings. The aim of this paper is to describe an augmented Lagrangian
method for the solution of the constrained optimization problem
\begin{equation}\label{Eq:Opt}\tag{$P$}
   \min\ f(x) \quad\text{subject to (s.t.)}\quad g(x)\le 0.
\end{equation}
We assume that $Y\hookrightarrow L^2(\Omega)$ densely for some measure space
$\Omega$, where the natural order on $L^2(\Omega)$ induces the order on $Y$. A
detailed description together with some remarks about this setting is given in
Section~2.

Augmented Lagrangian methods for the solution of optimization problems
belong to the most famous and successful algorithms for the
solution of finite-dimensional problems and are
described in almost all text books on continuous optimization, see, e.g.\
\cite{Bertsekas1995,Nocedal2006}. Their
generalization to infinite-dimensional problems has received considerable 
attention throughout the last decades \cite{Bergounioux1993,Bergounioux1997,
Fortin1983,Hintermueller2006,Ito1990a,Ito1990,Ito2000,Ito2008,Iusem2002}. However,
most existing approaches either assume a very specific problem structure
\cite{Bergounioux1993,Bergounioux1997}, require strong convexity assumptions
\cite{Hintermueller2006} or consider only the case where $Y$ is finite-dimensional 
\cite{Ito1990a,Ito2008}.

The contribution of the present paper is to overcome these limitations and
to provide a general convergence theory for infinite-dimensional problems.
To this end, we extend some of the recent contributions on the convergence of
certain modified augmented Lagrangian methods from the finite- to the 
infinite-dimensional case, cf.\ \cite{Birgin2014} and references therein 
for more details regarding some of the newer convergence results in the 
finite-dimensional setting. The main difference between the classical
augmented Lagrangian approach and its modified version consists
of a more controlled way of the multiplier update which is responsible
for a stronger global convergence theory.

Clearly, the main application of our theoretical framework \eqref{Eq:Opt} is
constrained optimization in function spaces, where the inequality constraint
with $Y\hookrightarrow L^2(\Omega)$ arises naturally. In particular, our
theory covers obstacle-type problems as well as optimal control problems
(including semilinear partial differential equations) with state constraints.

Let us remark that our algorithm can also be viewed as an extension of the quadratic penalty method (also called Moreau-Yosida regularization in the infinite-dimensional literature, e.g.\ \cite{Hintermueller2006,Ulbrich2011}). A numerical comparison to this method is given in Section \ref{Sec:Applic}.

This paper is organized as follows. In Section 2, we give a detailed overview
of our problem setting and assumptions. Section 3 contains a precise statement
of the algorithm, and we conduct a convergence analysis dedicated to global
optimization in Section 4. Starting with Section 5, we assume that the mappings
which constitute our problem are continuously differentiable, and establish some
theoretical foundations regarding KKT conditions and constraint qualifications.
In Section 6, we apply these insights to our algorithm and deduce corresponding
convergence results. Finally, Section 7 contains practical applications and we
conclude with some final remarks in Section 8.

\textbf{Notation:} We use standard notation such as $\dual{\cdot,\cdot}$ for the
duality pairing on $Y$, $\scal{\cdot,\cdot}$ for the scalar product in the Hilbert space
$Z$, and $\perp$ to denote orthogonality in $Z$. Moreover, $\mathcal{L}(X,Y)$ denotes the
space of continuous linear operators from $X$ into $Y$. The norms on $X$, $Y$, etc.\ are
denoted by $\|\cdot\|$, where an index (as in $\|\cdot\|_X$) is appended if necessary.
Furthermore, we write $\to$, $\wto$, and $\wto^*$ for strong, weak, and weak-$^*$
convergence, respectively. Finally, we use the abbreviation lsc for a lower
semicontinuous function.

\section{Preliminaries and Assumptions}\label{Sec:Prelims}

We denote by $e:Y\to Z$ the (linear and continuous) dense embedding of $Y$ into
$Z:=L^2(\Omega)$, and by $K_Y$, $K_Z$ the respective nonnegative cones in
$Y$ and $Z$, i.e.\
\begin{equation*}
    K_Z:=\{z\in Z\mid z(t)\ge 0~\text{a.e.}\} \quad\text{and}\quad
    K_Y:= \{ y\in Y \mid e(y) \in K_Z\}.
\end{equation*}
Note that the adjoint mapping $e^*$ embeds $Z^*$ into $Y^*$. Hence, we have
the chain
\begin{equation}\label{Eq:Embedding}
    Y \hookrightarrow Z \cong Z^* \hookrightarrow Y^*,
\end{equation}
which is occasionally referred to as a Gelfand triple.
The main reason for the specific configuration of our spaces $Y$ and $Z$ is
that the order on $Z=L^2(\Omega)$ has some structural properties which
may not hold on $Y$. For instance, the $L^2$-norm satisfies the relation
\begin{equation}\label{Eq:NormMonotone}
    0 \le z_1 \le z_2 \quad \implies \quad \|z_1\|_Z \le \|z_2\|_Z,
\end{equation}
which does not hold for, say, the spaces $H^1$ or $H_0^1$. (Note that
\eqref{Eq:NormMonotone} is one of the defining properties of so-called
Banach or Hilbert lattices \cite{Baiocchi1984,Yosida1995}. In fact,
$Z=L^2(\Omega)$ is a Hilbert lattice, but $H^1(\Omega)$ and $H_0^1(\Omega)$
are not.) We will put the properties of $Z$ to fruitful use by performing
the augmentation which constitutes our algorithm in $Z$. To simplify this,
we denote by $z_+$ and $z_-$ the positive and negative parts of $z\in Z$,
i.e.
\begin{equation*}
    z_+ := \max\{z,0\} \quad\text{and}\quad
    z_- := \max\{-z,0\}.
\end{equation*}
These operations have a variety of useful properties. For instance, we have
$z=z_+-z_-$ and $z_+ \perp z_-$ for every $z\in Z$.

Recall that, as in the introduction, we are concerned with the optimization
problem
\begin{equation*}\tag{\ref*{Eq:Opt}}
    \min\ f(x) \quad\text{s.t.}\quad g(x)\le 0,
\end{equation*}
where $Y\hookrightarrow Z=L^2(\Omega)$. Here, the inequality $g(x)\le 0$ has to
be understood with respect to the order induced by the cone $K_Y$, which is
implicitly given by the order on $Z$ through the embedding $e$.

The following is a list of assumptions which we will use throughout this paper.

\begin{asm}[General assumptions on the problem 
setting]\label{Asm:General}\leavevmode
\begin{enumerate}[label=(A\arabic*),leftmargin=1.5cm]
   \item $f$ and $\|g_+\|_Z$ are weakly lower semicontinuous.\label{Asm:Lsc}
   \item $f$ and $g$ are continuously Fr\'echet-differentiable.\label{Asm:C1}
   \item $y\mapsto |y|$ is well-defined and continuous on $Y$.\label{Asm:Abs}
   \item The unit ball in $Y^*$ is weak-$^*$ sequentially 
      compact.\label{Asm:YWC}
\end{enumerate}
\end{asm}

\noindent
Most of the theorems we will encounter later use only a subset of these
assumptions. Hence, we will usually list the assumptions for each theorem
explicitly by referencing to the names \ref{Asm:Lsc}-\ref{Asm:YWC}.

One assumption which might require some elaboration is the weak lower
semicontinuity of $\|g_+\|_Z$. To this end, note that there are various
theorems which characterize the weak lower semicontinuity of convex functions,
e.g.\ \cite[Thm.\ 9.1]{Bauschke2011}. Hence, if $\|g_+\|$ is convex (which is true if $g$ is convex with respect to the order in $Y$), then
the (strong) lower semicontinuity of $g$ already implies the weak lower
semicontinuity. We conclude that (A1) holds, in particular, for every 
lsc.\ convex function $ f $ and any mapping $g\in\mathcal{L}(X,Y)$.

On a further note, the above remarks offer another criterion for the weak
lower semicontinuity of $\|g_+\|_Z$. Since $y\mapsto \|y_+\|$ obviously has
this property, we conclude that it is sufficient for $g$ to be weakly
(sequentially) continuous.

Regarding the space $Y$ which is embedded into $Z$, recall that 
\ref{Asm:Abs} assumed the operation $y\mapsto |y|$ to be well-defined and 
continuous on $Y$. (Note that this assumption holds automatically if
$ Y = Z $, but in many applications, $ Y $ is only a subset of $ Z $, cf.\
the first remark below.) Hence,
the same holds for the mappings $y_+$, $y_-$, $\min$, $\max$, etc., which may
be defined in terms of their counterparts on $Z$.

We now give some general remarks about the setting \eqref{Eq:Opt}.
\begin{itemize}
    \item Clearly, one motivation for this setting is the case where $\Omega$
    is a bounded domain in $\R^d$ and $Y$ is one of the spaces $H^1(\Omega)$,
    $H_0^1(\Omega)$, or $C(\bar{\Omega})$. Problems of this type will be our
    main application in Section 7. Note that \ref{Asm:Abs} is satisfied for
    these spaces, cf.\ \cite{Coffman1975,Moshe1979} for a proof in $H^1$.
    \item In theory, we could easily generalize our work by allowing $Z$ to
    be an arbitrary Hilbert lattice \cite{Baiocchi1984,Meyer1991,Schaefer1974,
    Yosida1995}. However, it turns out \cite[Cor.\ 2.7.5]{Meyer1991} that every
    Hilbert lattice is (isometrically and lattice) isomorphic to $L^2(\Omega)$
    for some measure space $\Omega$. Hence, this seemingly more general setting
    is already covered by ours.
    \item Related to the previous point, we note that our setting also covers
    the case $Y=\R^m$, which is a Hilbert lattice and can be identified with
    $L^2(\Omega)$ on the discrete measure space $\Omega=\{1,\ldots,m\}$.
\end{itemize}
We conclude this section by proving a lemma for later reference. Recall that
$\scal{\cdot,\cdot}$ denotes the scalar product in $Z=L^2(\Omega)$.

\begin{lem}\label{Lem:Infzero}
    Let $(a^k)$ and $(b^k)$ be bounded sequences in $Z$. Then
    $\min\{a^k,b^k\}\to 0$ implies $\scal{a^k,b^k}\to 0$.
\end{lem}
\begin{proof}
    It is easy to see that $\scal{a^k,b^k}=\scal{\min\{a^k,b^k\},\max\{a^k,b^k\}}$. Since $\min\{a^k,b^k\}\to 0$ and the sequence $(\max\{a^k,b^k\})$ is bounded, it follows that $\scal{a^k,b^k}\to 0$.
\end{proof}

\noindent
Note that the above lemma becomes false if we drop the boundedness of one
of the sequences. For instance, consider the case where $\Omega=\{1\}$ and
$Z=L^2(\Omega)$, which can be identified with $\R$. Then the sequences
$a^k=k$ and $b^k=1/k$ provide a simple counterexample.

\section{An Augmented Lagrangian Method}\label{Sec:Method}

This section gives a detailed statement of our augmented Lagrangian
method for the solution of the optimization problem \eqref{Eq:Opt}.
It is motivated by the finite-dimensional discussion in, e.g.,
\cite{Birgin2014} and differs from the traditional augmented
Lagrangian method as applied, e.g., in \cite{Fortin1983,Ito1990} to a class
of infinite-dimensional problems, in a more controlled
updating of the Lagrange multiplier estimates.

We begin by defining the augmented Lagrangian
\begin{equation}\label{Eq:AL}
   L_{\rho}:X\times Z\to\R, \quad L_{\rho}(x,\lambda):=
   f(x)+\frac{\rho}{2}
   \left\|\left( g(x)+\frac{\lambda}{\rho} \right)_+\right\|_Z^2.
\end{equation}
This enables us to formulate the following algorithm for the solution of
\eqref{Eq:Opt}, which is a variant of the (finite-dimensional) method from
\cite{Birgin2014} in the context of our optimization problem \eqref{Eq:Opt}.
In fact, formally, the method looks almost identical to the one from
\cite{Birgin2014}, but some of the notations related to the order in $Y$ or $Z$
have a different and more general meaning than those in the finite-dimensional
literature.

\begin{alg}[Augmented Lagrangian method]\label{Alg:ALM}\leavevmode
\begin{itemize}[font=\normalfont]
   \item[(S.0)] Let $(x^0,\lambda^0)\in X\times Z$, $\rho_0>0$,
      $\BddMul^{\max}\in K_Z$, $\gamma>1$, $\tau\in(0,1)$, and set $k=0$.
   \item[(S.1)] If $(x^k,\lambda^k)$ satisfies a suitable stopping criterion: STOP.
   \item[(S.2)] Choose $0\le \BddMul^k\le \BddMul^{\max}$ and compute an approximate
      solution $x^{k+1}$ of
      \begin{equation}\label{Eq:PenOpt}
         \min_{x\in X}\ L_{\rho_k}(x,\BddMul^k).
      \end{equation}
   \item[(S.3)] Set $\lambda^{k+1}:=\left(\BddMul^k+\rho_k g(x^{k+1})\right)_+$.
      If $k=0$ or
      \begin{equation}\label{Eq:RhoTest}
         \left\|\min\left\{-g(x^{k+1}),\frac{\BddMul^k}{\rho_k}\right\}\right\|_Z
         \le\tau \left\|\min\left\{-g(x^k),\frac{\BddMul^{k-1}}{\rho_{k-1}}
         \right\}\right\|_Z
      \end{equation}
      holds, set $\rho_{k+1}:=\rho_k$; otherwise, set $\rho_{k+1}:=\gamma\rho_{k}$.
   \item[(S.4)] Set $k\leftarrow k+1$ and go to \textnormal{(S.1)}.
\end{itemize}
\end{alg}

\noindent
Note that the case $k=0$ is considered separately in Step~3 for
formal reasons only since $\BddMul^{k-1}$ and $\rho_{k-1}$ are not defined
for this value of the iteration counter. In any case, the treatment
of this initial step has no influence on our convergence theory.

One of the most important aspects of the above algorithm is the sequence $(\BddMul^k)$.
Note that $\BddMul^k\le \BddMul^{\max}$ implies that $(\BddMul^k)$ is bounded in $Z$.
Apart from this boundedness, there is a certain degree of freedom in the choice of
$\BddMul^k$. For instance, we could always choose $\BddMul^k:=0$ and thus obtain
a simplified algorithm which is essentially a quadratic penalty method. Going a
little further, our method also includes the Moreau-Yosida regularization scheme
(see \cite{Hintermueller2006,Ulbrich2011} and Section \ref{Sec:Applic}) as a special
case, which arises if $(\BddMul^k)$ is chosen as a constant sequence. However, the
most natural choice, which also brings the method closer to traditional augmented
Lagrangian schemes, is $\BddMul^k:=\min\{\lambda^k,\BddMul^{\max}\}$. That is,
$\BddMul^k$ is a bounded analogue of the possibly unbounded multiplier $\lambda^k$.

Another part of Algorithm \ref{Alg:ALM} which needs some explanation is our
notion of an ``approximate solution'' in Step~2. The reason we have not
specified this part is because we will carry out two distinct convergence
analyses which each require different assumptions.

\section{Global Minimization}\label{Sec:Global}

We begin by considering Algorithm \ref{Alg:ALM} from a global optimization
perspective. Note that most of the analysis in this section can be carried
out in the more general case where $f$ is an extended real-valued function,
i.e.\ $f$ maps to $ \R\cup\{+\infty\}$.

The global optimization perspective is particularly valid for convex
problems, where we can expect to solve the subproblems in Step 2 in a
global sense. This is reflected in the following assumption, which we
require throughout this section.

\begin{asm}\label{Asm:Global}
In Step~2 of Algorithm~\ref{Alg:ALM}, we obtain $x^{k+1}$ such that there
is a sequence $\varepsilon_k\downarrow 0$ with $L_{\rho_k}(x^{k+1},\BddMul^k)
\le L_{\rho_k}(x,\BddMul^k)+\varepsilon_k$ for all $x\in X$ and $k\in\N$.
\end{asm}

\noindent
This assumption is quite natural and basically asserts that
we finish each inner iteration with a point that is (globally) optimal within
some tolerance $\varepsilon_k$, and that this tolerance vanishes asymptotically.

Apart from Assumption \ref{Asm:Global}, the main requirement for the following
theorem is the weak lower semicontinuity of $f$ and $\|g_+\|_Z$, cf.\ \ref{Asm:Lsc}.
Note that $\|g_+\|_Z$ being weakly lsc implies a slightly stronger statement.
If $x^k\wto x$ and $z^k\to 0$ in $Z$, then the nonexpansiveness of $z\mapsto z_+$
(which is just the projection onto $K_Z$) together with \ref{Asm:Lsc} implies that
\begin{equation}\label{Eq:lscConsequence}
    \liminf_{k\to\infty}\big\|(g(x^k)+z^k)_+\big\|_Z=
    \liminf_{k\to\infty}\big\|g_+(x^k)\big\|_Z \ge \|g_+(x)\|_Z.
\end{equation}
This fact will be used in the proof of the following theorem.

\begin{thm}\label{Thm:OptimalityG}
Suppose that \ref{Asm:Lsc} and Assumption \ref{Asm:Global} hold. Let $(x^k)$
be a sequence generated by Algorithm~\ref{Alg:ALM}, and let $\bar{x}$ be a
weak limit point of $(x^k)$. Then:
\begin{enumerate}[label=\textnormal{(\alph*)}]
   \item $\bar{x}$ is a global minimum of the function $\|g_+(x)\|_Z^2$.
   \item If $\bar{x}$ is feasible, then $\bar{x}$ is a solution of the
      optimization problem \eqref{Eq:Opt}.
\end{enumerate}
\end{thm}

\begin{proof}
(a): We first consider the case where $(\rho_k)$ is bounded. Recalling
\eqref{Eq:RhoTest}, we obtain
\begin{equation*}
    \|g_+(x^{k+1})\|_Z\le
    \left\| \min \left\{ -g(x^{k+1}), \frac{\BddMul^k}{\rho_k} \right\} \right\|_Z 
    \to 0.
\end{equation*}
Hence \ref{Asm:Lsc} implies that $\bar{x}$ is feasible and the assertion 
follows trivially.

Next, we consider the case where $\rho_k\to\infty $. Let $\mathcal{K}\subset\N$
be such that $x^{k+1}\wto_{\mathcal{K}}\bar{x}$ and assume that there is an
$x\in X$ with $\|g_+(x)\|_Z^2<\|g_+(\bar{x})\|_Z^2$.
By \eqref{Eq:lscConsequence}, the boundedness of $(\BddMul^k)$, and the fact 
that $\rho_k\to\infty$, there is a constant $c>0$ such that
\begin{equation*}
   \left\|\left(g(x^{k+1})+\frac{\BddMul^k}{\rho_k}\right)_+\right\|_Z^2 >
   \left\|\left(g(x)+\frac{\BddMul^k}{\rho_k}\right)_+\right\|_Z^2+c
\end{equation*}
holds for all $k\in\mathcal{K}$ sufficiently large. Hence,
\begin{equation*}
   L_{\rho_k}(x^{k+1},\BddMul^k)>L_{\rho_k}(x,\BddMul^k)+\frac{\rho_k c}{2}+f(x^{k+1})-f(x).
\end{equation*}
Using Assumption~\ref{Asm:Global}, we arrive at the inequality
\begin{equation*}
    \varepsilon_k > \frac{\rho_k c}{2}+f(x^{k+1})-f(x),
\end{equation*}
where $\varepsilon_k\to 0$. Since $(f(x^{k+1}))_{\mathcal{K}}$ is bounded from
below by the weak lower semicontinuity of $f$, this is a contradiction.\medskip

\noindent
(b): Let $\mathcal{K}\subset\N$ be such that $x^{k+1}\wto_{\mathcal{K}}\bar{x}$,
and let $x$ be any other feasible point. From Assumption~\ref{Asm:Global},
we get
\begin{equation*}
   L_{\rho_k}(x^{k+1},\BddMul^k)\le L_{\rho_k}(x,\BddMul^k)+\varepsilon_k.
\end{equation*}
Again, we distinguish two cases. First assume that $ \rho_k \to \infty $.
By the definition of the augmented Lagrangian, we have (recall that $x$
is feasible)
\begin{equation*}
   f(x^{k+1})\le f(x)+\frac{\rho_k}{2} \left\|\left( g(x)+\frac{\BddMul^k}{\rho_k}
   \right)_+ \right\|_Z^2 +\varepsilon_k \le
   f(x)+\frac{\|\BddMul^k\|_Z^2}{2 \rho_k}+\varepsilon_k.
\end{equation*}
Taking limits in the above inequality, using the boundedness of $(\BddMul^k)$
and the weak lower semicontinuity of $f$, we get $f(\bar{x})\le f(x)$.

Next, consider the case where $(\rho_k)$ is bounded. Using the feasibility
of $x$ and a similar inequality to above, it follows that
\begin{equation*}
   f(x^{k+1})+\frac{\rho_k}{2} \left\|\left( g(x^{k+1})+
   \frac{\BddMul^k}{\rho_k} \right)_+ \right\|_Z^2 \le
   f(x)+\frac{\rho_k}{2}\left\|\frac{\BddMul^k}{\rho_k}\right\|_Z^2+\varepsilon_k.
\end{equation*}
But
\begin{equation*}
   \left( g(x^{k+1})+\frac{\BddMul^k}{\rho_k} \right)_+ = \frac{\BddMul^k}{\rho_k} -
   \min\left\{ -g(x^{k+1}), \frac{\BddMul^k}{\rho_k} \right\}
\end{equation*}
and the latter part tends to $0$ because of \eqref{Eq:RhoTest}.
This implies $f(\bar{x})\le f(x)$.
\end{proof}

\noindent
Note that, for part (a) of the theorem, we did not fully use $\varepsilon_k
\downarrow 0$; we only used the fact that $(\varepsilon_k)$ is bounded.
Hence, this result remains true under weaker conditions than those given
in Assumption~\ref{Asm:Global}. Furthermore, note that
Theorem~\ref{Thm:OptimalityG} does not require any differentiability
assumption, though, in practice, the approximate solution of the subproblems
in (S.2) of Algorithm~\ref{Asm:Global} might be easier under differentiability
assumptions. Finally, note that, in view of statement (a), the weak
limit point $ \bar x $ is always feasible if the feasible set of
the optimization problem \eqref{Eq:Opt} is nonempty, i.e.\ in this case
the feasibility assumption from statement (b) is always satisfied.
On the other hand, if the feasible set is empty, it is interesting to
note that statement (a) still holds, whereas the assumption from
statement (b) cannot be satisfied.

We now turn to a convergence theorem which guarantees, under certain assumptions,
the (strong) convergence of the whole sequence $(x^k)$. Such an assertion usually
requires a suitable convexity or second-order condition and, in fact, there are
many results along this line in the context of augmented Lagrangian methods,
e.g.\ \cite{Bertsekas1982,Fernandez2012} in finite or \cite{Hintermueller2006}
in infinite dimensions. Here, we prove a theorem which shows that our method
converges globally for convex problems where the objective function is strongly
convex. Note that, in the convex setting, the lower semicontinuity assumption
\ref{Asm:Lsc} is fairly weak since it is always implied by (ordinary) continuity.
Moreover, let us emphasize that the theorem below does not require any Lagrange
multiplier or constraint qualification.

\begin{thm}\label{Thm:Convex}
    Suppose that \ref{Asm:Lsc} and Assumption~\ref{Asm:Global} hold, and that $X$ is
    reflexive. Furthermore, assume that $g$ is convex, that $f$ is strongly
    convex, and that the feasible set of \eqref{Eq:Opt} is nonempty. Then \eqref{Eq:Opt}
    admits a unique solution $\bar{x}$, and the sequence $(x^k)$ from
    Algorithm~\ref{Alg:ALM} converges (strongly) to $\bar{x}$.
\end{thm}

\begin{proof}
    Under the given assumptions, it is easy to show that $f$ is coercive and that
    the feasible set of \eqref{Eq:Opt} is closed and convex, hence weakly closed.
    Therefore, existence and uniqueness of the solution $\bar{x}$ follow from
    standard arguments.
    
    Now, denoting by $c>0$ the modulus of convexity of $f$, it follows that
    \begin{equation}\label{Eq:ThmConvex1}
        \frac{c}{8} \|x^{k+1}-\bar{x}\|_X^2 \le \frac{f(x^{k+1})+f(\bar{x})}{2}-
        f\mleft( \frac{x^{k+1}+\bar{x}}{2} \mright)
    \end{equation}
    for all $k$. By the proof of Theorem \ref{Thm:OptimalityG} (b), it is easy to
    see that $\limsup_{k\to\infty}f(x^{k+1})\le f(\bar{x})$. Hence, taking into
    account that $f$ is bounded from below, it follows from \eqref{Eq:ThmConvex1}
    that $(x^k)$ is bounded. Since $X$ is reflexive and every weak limit point
    of $(x^k)$ is a solution of \eqref{Eq:Opt} by Theorem \ref{Thm:OptimalityG},
    we conclude that $x^k\wto \bar{x}$. In particular, the weak lower
    semicontinuity of $f$ together with $\limsup_{k\to\infty}f(x^{k+1})\le f(\bar{x})$
    implies that $f(x^{k+1})\to f(\bar{x})$. Moreover, we also have
    $f(\bar{x})\le \liminf_{k\to\infty}f\bigl( (x^k+\bar{x})/2 \bigr)$. Hence,
    \eqref{Eq:ThmConvex1} implies $\|x^{k+1}-\bar{x}\|_X\to 0$.
\end{proof}

\noindent
The above theorem also shows that strong convergence of the primal iterates is not
completely unrealistic, even in infinite dimensions. It may be possible to also prove
strong convergence by using some local condition (e.g.\ a suitable second-order
sufficient condition in a stationary point). However, we will now explore this
subject any further in the present paper.

\section{Sequential KKT conditions}\label{Sec:KKT}

Throughout this section, we assume that $f$ and $g$ are continuously 
Fr\'echet-differentiable on $X$, and discuss the KKT conditions of the
optimization problem \eqref{Eq:Opt}. Recalling that $K_Y$ is the nonnegative
cone in $Y$, we denote by
\begin{equation*}
    K_Y^+ := \{ f\in Y^* \mid \dual{f,y}\ge 0~
    \forall y\in K_Y \}
\end{equation*}
its dual cone. This enables us to define the KKT conditions as follows.

\begin{dfn}\label{Dfn:KKT}
A tuple $(x,\lambda)\in X\times K_Y^+$ is called a \emph{KKT point} of
\eqref{Eq:Opt} if
\begin{equation}\label{Eq:KKT}
   f'(x)+g'(x)^*\lambda=0, \quad g(x)\le 0, \quad\text{and}\quad
   \dual{\lambda,g(x)}=0.
\end{equation}
We also call $x\in X$ a \emph{KKT point} of $\eqref{Eq:Opt}$ if
$(x,\lambda)$ is a KKT point for some $\lambda$.
\end{dfn}

\noindent
From a practical perspective, when designing an algorithm for the solution
of \eqref{Eq:Opt}, we will expect the algorithm to generate a sequence which
satisfies the KKT conditions in an asymptotic sense. Hence, it will be
extremely important to discuss a sequential analogue of the KKT conditions.

\begin{dfn}\label{Dfn:AKKT}
We say that the \emph{asymptotic KKT (or AKKT) conditions} hold in a feasible
point $x\in X$ if there are sequences $x^k\to x$ and $(\lambda^k)\subset K_Y^+$
such that
\begin{equation}\label{Eq:AKKT}
   \quad f'(x^k)+g'(x^k)^* \lambda^k\to 0 \quad\text{and}\quad
   \dual{\lambda^k,g_-(x^k)}\to 0.
\end{equation}
\end{dfn}

\noindent
Asymptotic KKT-type conditions have previously been considered in the 
literature \cite{Andreani2011,Andreani2010,Birgin2014} for 
finite-dimensional optimization problems. Furthermore, 
in \cite{Birgin2014}, it is shown that AKKT is a necessary
optimality condition even in the absence of constraint qualifications. With
little additional work, this result can be extended to our infinite-dimensional
setting.

\begin{thm}\label{Thm:AKKT}
Suppose that \ref{Asm:Lsc}, \ref{Asm:C1} hold, and that $X$ is reflexive. Then
every local solution $\bar{x}$ of \eqref{Eq:Opt} satisfies the AKKT conditions.
\end{thm}

\begin{proof}
To simplify the proof, we assume that the squared norm $q(x):=\|x\|_X^2$ is continuously differentiable on $X$; since $X$ is reflexive, this is no restriction as $X$ can be renormed equivalently with
a continuously differentiable norm \cite{Fry2002,Troyanski1970}. By assumption, there is an $r>0$ such that $\bar{x}$ solves \eqref{Eq:Opt} on $B_r(\bar{x})$. Now, for $k\in\N$, consider the problem
\begin{equation}\label{Eq:ThmAKKT1}
    \min\ f(x)+k\|g_+(x)\|_Z^2+\|x-\bar{x}\|_X^2 \quad\text{s.t.}\quad
    x\in B_r(\bar{x}).
\end{equation}
Since the above objective function is weakly lsc and $B_r(\bar{x})$ is weakly
compact, this problem admits a solution $x^k$. Due to $(x^k)\subset B_r(\bar{x})$,
there is a $\mathcal{K}\subset\mathbb{N}$ such that $x^k\wto_{\mathcal{K}}\bar{y}$
for some $\bar{y}\in B_r(\bar{x})$. Since $x^k$ is a solution of
\eqref{Eq:ThmAKKT1}, we have
\begin{equation}\label{Eq:ThmAKKT2}
    f(x^k)+k\|g_+(x^k)\|_Z^2+\|x^k-\bar{x}\|_X^2 \le f(\bar{x})
\end{equation}
for every $k$. Dividing by $k$ and taking the limit $k\to_{\mathcal{K}}\infty$,
we obtain from \ref{Asm:Lsc} that $\|g_+(\bar{y})\|_Z=0$, i.e.\ $\bar{y}$ is 
feasible. By \eqref{Eq:ThmAKKT2}, we also obtain 
$f(\bar{y})+\|\bar{y}-\bar{x}\|_X^2\le f(\bar{x})$. But $f(\bar{x})\le f(\bar{y})$, hence $\bar{x}=\bar{y}$ and \eqref{Eq:ThmAKKT2} implies that $x^k\to_{\mathcal{K}}\bar{x}$.
In particular, we have $\|x^k-\bar{x}\|_X<r$ for sufficiently large 
$k\in\mathcal{K}$, and from \eqref{Eq:ThmAKKT1} we obtain
\begin{equation*}
    f'(x^k) + 2 k g'(x^k)^* g_+(x^k) + q' (x^k-\bar{x})=0.
\end{equation*}
Define $\lambda^k:=2 k g_+(x^k)$. Then $f'(x^k)+g'(x^k)^*\lambda^k
\to_{\mathcal{K}} -q'(0)=0$ and $\dual{\lambda^k,g_-(x^k)}=0$.
\end{proof}

\noindent
The above theorem also motivates our definition of the AKKT conditions. In
particular, it justifies the formulation of the complementarity condition as
$\dual{\lambda^k,g_-(x^k)}\to 0$, since the proof shows that $(\lambda^k)$
need not be bounded. Hence, the conditions
\begin{equation*}
    \min \{-g(x^k),\lambda^k\}\to_{\mathcal{K}}0, \quad
    \dual{\lambda^k,g(x^k)}\to_{\mathcal{K}}0,
    \quad\text{and}\quad \dual{\lambda^k,g_-(x^k)}\to_{\mathcal{K}} 0
\end{equation*}
are not equivalent. Note that the second of these conditions (which might
appear as the most natural formulation of the complementarity condition) is
often violated by practical algorithms \cite{Andreani2010}.

In order to get the (clearly desirable) implication ``AKKT $\Rightarrow$ KKT'',
we will need a suitable constraint qualification. In the finite-dimensional
setting, constraint qualifications such as MFCQ and 
CPLD \cite{Birgin2014,Qi2000}
have been used to enable this transition. However, in the infinite-dimensional
setting, our choice of constraint qualification is much more restricted. For
instance, we are not aware of any infinite-dimensional analogues of the (very
amenable) CPLD condition. Hence, we have decided to employ the Zowe-Kurcyusz
regularity condition \cite{Zowe1979}, which is known to be equivalent to the
Robinson condition \cite{Robinson1976} and to be a generalization of the
finite-dimensional MFCQ. It should be noted, however, that any condition
which guarantees ``AKKT $\Rightarrow$ KKT'' could be used in our analysis.

\begin{dfn}\label{Dfn:ZoweKurcyusz}
The \emph{Zowe-Kurcyusz condition} holds in a feasible point $x\in X$ if
\begin{equation}\label{Eq:ZoweKurcyusz}
   g'(x)X+\operatorname{cone}(K_Y+g(x))=Y,
\end{equation}
where $\operatorname{cone}(K_Y+g(x))$ is the conical hull of $K_Y+g(x)$ in $Y$.
\end{dfn}

\noindent
We note that the complete theory in this paper can be written down with
$ Y = Z $ only, so, formally, there seems to be no reason for introducing
the imbedded space $ Y $. One of the main reasons for the more general
framework considered here with an additional space $ Y $ is that 
suitable constraint qualifications like the above Zowe-Kurcyusz condition
are typically violated even in simple applications when formulated in 
$ Z $, whereas we will see in Section 7 that this condition easily holds in
suitable spaces $ Y $. We therefore stress the importance of 
Definition~\ref{Dfn:ZoweKurcyusz} being defined in $ Y $, and not in $ Z $.

Let us also remark that the applicability of the Zowe-Kurcyusz condition
very much depends on the particular structure of the constraints. For many
simple constraints, the operator $g'(x)$ is actually surjective, which implies
that \eqref{Eq:ZoweKurcyusz} holds trivially. This is the case, for instance,
when dealing with one-sided box constraints in a suitable function space.
(The case of two-sided box constraints is more difficult; see, e.g.,
\cite{Troeltzsch2010}.) Another case where the Zowe-Kurcyusz condition holds
automatically is if $K_Y$ has a nonempty interior and the problem satisfies
a linearized Slater condition, e.g.\ \cite[Eq.\ 6.18]{Troeltzsch2010}.

One of the most important consequences of the Zowe-Kurcyusz condition
is that the set of multipliers corresponding to a KKT point $x$ is bounded
\cite{Zowe1979}. From this point of view, it is natural to expect that
the sequence $(\lambda^k)$ from Definition \ref{Dfn:AKKT} is bounded,
provided the limit point $x$ satisfies the Zowe-Kurcyusz condition.

\begin{thm}\label{Thm:ZoweKurcyusz}
Suppose that \ref{Asm:C1} holds. Let $x\in X$ be a point which satisfies the
AKKT conditions, and let $(x^k)$, $(\lambda^k)$ be the corresponding sequences
from Definition~\ref{Dfn:AKKT}.
\begin{enumerate}[label=\textnormal{(\alph*)}]
   \item If $x$ satisfies the Zowe-Kurcyusz condition, then $(\lambda^k)$
      is bounded in $Y^*$.
   \item If \ref{Asm:Abs}, \ref{Asm:YWC} hold and $(\lambda^k)$ is bounded
   in $Y^*$, then $x$ is a KKT point.
\end{enumerate}
\end{thm}

\begin{proof}
(a): In view of \cite[Thm.\ 2.1]{Zowe1979}, the Zowe-Kurcyusz condition
implies that there is an $r>0$ such that
\begin{equation*}
   B_r^Y \subset g'(x) B_1^X + (K_Y+g(x)) \cap B_1^Y,
\end{equation*}
where $B_r^X$ and $B_r^Y$ are the closed $r$-balls around zero in $X$ and $Y$,
respectively. By the AKKT conditions and \ref{Asm:C1}, there is a 
$k_0\in\N$ such that
\begin{equation*}
    \|g(x^k)-g(x)\|_Y\le \frac{r}{4}
    \quad\text{and}\quad
    \|g'(x^k)-g'(x)\|_{\mathcal{L}(X,Y)}\le
    \frac{r}{4}
\end{equation*}
for every $k\ge k_0$. Now, let $u\in B_r^Y$ and $k\ge k_0$. It follows that
$-u=g'(x)w+z$ with $\|w\|_X\le 1$ and $z=z_1+g(x)$, $\|z\|_Y\le 1$,
$z_1\in K_Y$. Furthermore, the AKKT conditions imply
\begin{align*}
    \dual{\lambda^k,z_1+g(x^k)}
    & = \dual{\lambda^k,z_1} +\dual{\lambda^k,g_+(x^k)} - \dual{\lambda^k,g_-(x^k)} \\
    & \ge -\dual{\lambda^k,g_-(x^k)} \\
    & \to 0,
\end{align*}
i.e.\ $\dual{\lambda^k,z_1+g(x^k)}$ is bounded from below. Using once again the
AKKT conditions, we see that $\dual{\lambda^k,g'(x^k)w}$ is also bounded, and
it follows that
\begin{align*}
    \dual{\lambda^k,u} & = - \dual{\lambda^k,g'(x)w} - \dual{\lambda^k,z_1+g(x)} \\
    & \le \frac{r}{4}\|\lambda^k\|_{Y^*} - \dual{\lambda^k,g'(x^k)w} +
    \frac{r}{4}\|\lambda^k\|_{Y^*} - \dual{\lambda^k,z_1+g(x^k)} \\
    & \le \frac{r}{2}\|\lambda^k\|_{Y^*} + C
\end{align*}
for some constant $C>0$. We conclude that
\begin{equation*}
   \|\lambda^k\|_{Y^*}=\sup_{\|u\|\le r}\dual{\lambda^k,\frac{1}{r}u}\le \frac{1}{r}
   \left(C+\frac{r}{2}\|\lambda^k\|_{Y^*}\right)
\end{equation*}
and, hence, $\|\lambda^k\|_{Y^*}\le 2C/r$.\medskip

\noindent
(b): Since $(\lambda^k)$ is bounded in $Y^*$ and the unit ball in $Y^*$ is
weak-$^*$ sequentially compact by \ref{Asm:YWC}, 
there is a $\mathcal{K}\subset\N$ such that
$\lambda^k\wto_{\mathcal{K}}^*\lambda$ for some $\lambda\in Y^*$. Using
$\lambda^k \in K_Y^+$ for all $k\in\N$, it follows that
\begin{equation*}
    \dual{\lambda,y}=\lim_{k\in\mathcal{K}}\dual{\lambda^k,y}\ge 0
\end{equation*}
for every $y\in K_Y$. In other words, $\lambda\in K_Y^+$. Hence, taking the
limit in the AKKT conditions and using $g_-(x^k)\to g_{-}(x) = g(x)$ in $Y$,
which is a consequence of \ref{Asm:Abs} and the feasibility of $ x $, we
see that $(x,\lambda)$ satisfies the KKT conditions.
\end{proof}

\noindent
The above theorem is a generalization of a well-known result for the
MFCQ constraint qualification in finite dimensions. Recall that, for
$Y=\R^m$ with the natural ordering, the Zowe-Kurcyusz condition is
equivalent to MFCQ \cite{Zowe1979}.

\section{Convergence to KKT Points}\label{Sec:Convergence}

We now discuss the convergence properties of Algorithm~\ref{Alg:ALM}
from the perspective of KKT points. To this end, we make the following
assumption.

\begin{asm}\label{Asm:AKKT}
In Step~2 of Algorithm~\ref{Alg:ALM}, we obtain $x^{k+1}$ such that
$L_{\rho_k}'(x^{k+1},\BddMul^k)\to 0$ as $k\to\infty$.
\end{asm}

\noindent
The above is a very natural assumption which states that $x^{k+1}$ is an
(approximate) stationary point of the respective subproblem. Note that,
from \eqref{Eq:AL}, we obtain the following formula for the derivative of
$L_{\rho_k}$ with respect to $x$:
\begin{align*}
   L_{\rho_k}'(x^{k+1},\BddMul^k)
   & = f'(x^{k+1})+g'(x^{k+1})^* (\BddMul^k+\rho_k g(x^{k+1}))_+ \\
   & = f'(x^{k+1})+g'(x^{k+1})^* \lambda^{k+1}.
\end{align*}

\noindent
Our further analysis is split into a discussion of feasibility and optimality.
Regarding the feasibility aspect, note that we can measure the infeasibility
of a point $x$ by means of the function $\|g_+(x)\|_Z^2$. By standard projection
theorems, this is a Fr\'echet-differentiable function and its derivative is 
given by $D\|g_+(x)\|_Z^2=2 g'(x)^* g_+(x)$, cf.\ 
\cite[Cor.\ 12.31]{Bauschke2011}. This will be used in the proof of
the following theorem.

\begin{thm}\label{Thm:Feasibility}
Suppose that \ref{Asm:C1} and Assumption \ref{Asm:AKKT} hold. If $(x^k)$ is
generated by Algorithm~\ref{Alg:ALM} and $\bar{x}$ is a limit point of $(x^k)$,
then $D\|g_+(\bar{x})\|_Z^2=0$.
\end{thm}

\begin{proof}
Let $\mathcal{K}\subset\N$ be such that $x^{k+1}\to_{\mathcal{K}} \bar{x}$.
If $(\rho_k)$ is bounded, then we can argue as in the proof of Theorem
\ref{Thm:OptimalityG} (a) and conclude that $\bar{x}$ is feasible. Hence,
there is nothing to prove. Now, assume that $\rho_k\to\infty$. By
Assumption~\ref{Asm:AKKT}, we have
\begin{equation*}
   f'(x^{k+1})+g'(x^{k+1})^* (\BddMul^k+\rho_k g(x^{k+1}))_+\to 0.
\end{equation*}
Dividing by $\rho_k$ and using the boundedness of $(\BddMul^k)$ and
$(f'(x^{k+1}))_{\mathcal{K}}$, it follows that
\begin{equation*}
   g'(x^{k+1})^* g_+(x^{k+1})\to_{\mathcal{K}} 0.
\end{equation*}
This completes the proof.
\end{proof}

\noindent
Similarly to Theorem \ref{Thm:OptimalityG}, we remark that the above result
does not fully use the fact that $L_{\rho_k}'(x^{k+1},\BddMul^k)\to 0$ and remains
valid if this sequence is only bounded.

We now turn to the optimality of limit points of Algorithm~\ref{Alg:ALM}.
To this end, recall that Assumption~\ref{Asm:AKKT} implies that
\begin{equation}\label{Eq:OptimalityAKKT}
   f'(x^{k+1})+g'(x^{k+1})^*\lambda^{k+1}\to 0,
\end{equation}
which already suggests that the sequence of tuples $(x^k,\lambda^k)$ satisfies
AKKT for the optimization problem \eqref{Eq:Opt}. In fact, the only missing
ingredient is the asymptotic complementarity of $g$ and $\lambda$. We deal
with this issue in two steps. First, we consider the case where $(\rho_k)$ is
bounded. In this case, we even obtain the (exact) KKT conditions without any
further assumptions.

\begin{thm}\label{Thm:OptimalityB}
Suppose that \ref{Asm:C1} and Assumption \ref{Asm:AKKT} hold. Let $(x^k)$ be
generated by Algorithm~\ref{Alg:ALM} and assume that $(\rho_k)$ is bounded.
If $\bar{x}$ is a limit point of $(x^k)$, then $\bar{x}$ satisfies the KKT conditions of \eqref{Eq:Opt} with a multiplier in $Z$.
\end{thm}

\begin{proof}
Let $\mathcal{K}\subset\N$ be such that $x^{k+1}\to_{\mathcal{K}} \bar{x}$.
Without loss of generality, we assume that $\rho_k=\rho_0$ for all $k$. From
Algorithm~\ref{Alg:ALM}, it follows that $(\lambda^{k+1})_{\mathcal{K}}$
is bounded in $Z$ and
\begin{equation*}
   \min\left\{-g(x^{k+1}),\frac{\BddMul^k}{\rho_0}\right\}\to 0 \quad\text{in}~Z.
\end{equation*}
As in the proof of Theorem~\ref{Thm:OptimalityG} (a), this implies
$\|g_+(x^{k+1})\|_Z\to 0$. Furthermore, from Lemma~\ref{Lem:Infzero}, we
get $\scal{\BddMul^k,g(x^{k+1})}\to_{\mathcal{K}} 0$. Using the definition of
$\lambda^{k+1}$, we now obtain
\begin{align*}
   \scal{\lambda^{k+1},g(x^{k+1})}
   & =\rho_0 \scal{\left(\frac{\BddMul^k}{\rho_0}+g(x^{k+1})\right)_+,g(x^{k+1})} \\
   & =\rho_0 \scal{\frac{\BddMul^k}{\rho_0}-\min\left\{
   -g(x^{k+1}), \frac{\BddMul^k}{\rho_0} \right\}, g(x^{k+1})} \\
   & \to_{\mathcal{K}} 0.
\end{align*}
Since $(\lambda^{k+1})_{\mathcal{K}}$ is bounded in $Z$, this also implies
$\scal{\lambda^{k+1},g_-(x^{k+1})}\to_{\mathcal{K}} 0$. Hence, recalling
\eqref{Eq:OptimalityAKKT}, the AKKT conditions hold in $\bar{x}$.
Now, the claim essentially follows from Theorem~\ref{Thm:ZoweKurcyusz} (b),
the only difference here is that we are working in the Hilbert space
$ Z $ instead of $ Y $ or $ Y^* $, hence the two conditions \ref{Asm:Abs} and
\ref{Asm:YWC} formally required in Theorem~\ref{Thm:ZoweKurcyusz} (b) are
automatically satisfied in the current Hilbert space situation.
\end{proof}

\noindent
Some further remarks about the case of bounded multipliers are due. In
this case, the multiplier sequence $(\lambda^k)_{\mathcal{K}}$ is also bounded in $Z$,
and it does not make a difference whether we state the asymptotic
complementarity of $g(x^k)$ and $\lambda^k$ as
\begin{equation*}
   \min \{-g(x^k),\lambda^k\}\to_{\mathcal{K}}0, \quad
   \scal{\lambda^k,g(x^k)}\to_{\mathcal{K}}0,
   \quad\text{or}\quad \scal{\lambda^k,g_-(x^k)}\to_{\mathcal{K}} 0,
\end{equation*}
cf.\ the remarks after Theorem \ref{Thm:AKKT}.
However, this situation changes
if we turn to the case where $(\rho_k)$ is unbounded. Here, it is essential
that we define the asymptotic KKT conditions exactly as we did in Definition
\ref{Dfn:AKKT}.

\begin{thm}\label{Thm:OptimalityU}
Suppose that \ref{Asm:C1}, \ref{Asm:Abs} and Assumption \ref{Asm:AKKT} hold.
Let $(x^k)$ be generated by Algorithm~\ref{Alg:ALM} and let $\rho_k\to\infty$.
Then every limit point $\bar{x}$ of $(x^k)$ which is feasible satisfies AKKT
for the optimization problem \eqref{Eq:Opt}.
\end{thm}

\begin{proof}
Let $\mathcal{K}\subset\N$ be such that $x^{k+1}\to_{\mathcal{K}}\bar{x}$. Recalling \eqref{Eq:OptimalityAKKT}, it suffices to show that
\begin{equation*}
   \left(\lambda^{k+1},g_-(x^{k+1})\right)
   =\left((\BddMul^k+\rho_k g(x^{k+1}))_+,g_-(x^{k+1})\right)\to_{\mathcal{K}} 0.
\end{equation*}
To this end, let  $v^k:=(\BddMul^k+\rho_k g(x^{k+1}))_+ \cdot
g_-(x^{k+1})\in L^1(\Omega)$. Since $v^k\ge 0$, we show that $\int_\Omega v^k \to 0$
by using the dominated convergence theorem. Subsequencing if necessary,
we may assume that $g(x^{k+1})$ converges pointwise to $g(\bar{x})$ and that
$|g(x^{k+1})|$ is bounded a.e.\ by an $L^2$-function for $k\in\mathcal{K}$
\cite[Thm.\ 4.9]{Brezis2011}. It follows that $v^k$ is
bounded a.e.\ by an $L^1$-function (recall that, if $g(x^{k+1})(t)\ge 0$, then
$v^k(t)=0$). Hence, we only need to show that $v^k(t)\to 0$ pointwise.
Let $t\in \Omega$ and distinguish two cases:\\
\textbf{Case 1.} $g(\bar{x})(t)<0$. In this case, the pointwise convergence
implies that $\BddMul^k(t)+\rho_k g(x^{k+1})(t)<0$ for sufficiently large $k$ and,
hence, $v^k(t)=0$ for all such $k$.\\
\textbf{Case 2.} $g(\bar{x})(t)=0$. Then consider a fixed $ k \in
\mathcal{K} $. If $g(x^{k+1})(t)\ge 0$, it follows again from the definition
of $v^k$ that $v^k(t)=0$. On the other hand, if $g(x^{k+1})(t)<0$, it follows
that $v^k(t)\le \BddMul^k(t)\cdot|g(x^{k+1})(t)|$, and the right-hand side
converges to zero if this subcase occurs infinitely many times.

Summarizing these two cases, the pointwise convergence $v^k(t)\to 0$
follows immediately. The assertion is therefore a consequence of
the dominated convergence theorem.
\end{proof}

\noindent
For a better overview, we now briefly summarize the two previous
convergence theorems. To this end, let $(x^k)$ be the sequence generated by
Algorithm~\ref{Alg:ALM} and let $\bar{x}$ be a limit point of $(x^k)$. For the
sake of simplicity, we assume that \ref{Asm:C1}-\ref{Asm:YWC} hold. Then
Theorems \ref{Thm:OptimalityB}, \ref{Thm:OptimalityU} and \ref{Thm:ZoweKurcyusz}
imply that $\bar{x}$ is a KKT point if either
\begin{enumerate}[label=\textnormal{(\alph*)}]
   \item the sequence $(\rho_k)$ is bounded, or
   \item $\rho_k\to\infty$, $\bar{x}$ is feasible and the Zowe-Kurcyusz
      condition holds in $\bar{x}$.
\end{enumerate}
Hence, for $\rho_k\to\infty$, the success of the algorithm crucially
depends on the achievement of feasibility and the regularity of the
constraint function $g$. Recall that, by Theorem \ref{Thm:Feasibility},
the limit point $\bar{x}$ is always a stationary point of the constraint
violation $\|g_+(x)\|_Z^2$. Hence, situations in which $\bar{x}$ is infeasible
are rare; in particular, this cannot occur for convex problems (unless, of
course, the feasible set itself is empty).

\section{Applications}\label{Sec:Applic}

We now give some applications and numerical results for Algorithm \ref{Alg:ALM}. To this end, we consider some standard problems from the literature. Apart from the first example, we place special emphasis on nonlinear and nonconvex problems since the appropriate treatment of these is one of the focal points of our method.

All our examples follow the general pattern that $X$, $Y$, $Z$ are (infinite-dimensional) function spaces on some bounded domain $\Omega\subseteq\R^d$, $d\in\N$. In each of the subsections, we first give a general overview about the problem in question and then present some numerical results on the unit square $\Omega=(0,1)^2$.

In practice, Algorithm \ref{Alg:ALM} is then applied to a (finite-dimensional) discretization of the corresponding problem. Hence, we implemented the algorithm for finite-dimensional problems. The implementation was done in MATLAB\textsuperscript{\textregistered} and uses the parameters
\begin{equation*}
    \lambda^0:=0,\quad \rho_0:=1,\quad
    \BddMul^{\max}:=10^6 e,
    \quad \gamma:=10,\quad \tau:=0.1
\end{equation*}
(where $\lambda^0$, $\BddMul^{\max}$, and $ e := (1, \ldots, 1)^T $ are understood to be of appropriate dimension), together with a problem-dependent starting point $x^0$. The sequence $(\BddMul^k)$ is chosen as $\BddMul^k:=\min\{\lambda^k,\BddMul^{\max}\}$, i.e.\ it is a safeguarded analogue of the multiplier sequence. The overall stopping criterion which we use for our algorithm is given by
\begin{equation*}
    \left\| \nabla f(x) + \nabla g(x) \lambda \right\|_{\infty} \le 10^{-4}
    \quad\text{and}\quad
    \left\|\min\{-g(x),\lambda\}\right\|_{\infty} \le 10^{-4},
\end{equation*}
i.e.\ it is an inexact KKT condition. Furthermore, in each outer iteration, we solve the corresponding subproblem in Step 2 by computing a point $x^{k+1}$ which satisfies
\begin{equation*}
    \left\| L_{\rho_k}'(x^{k+1},\BddMul^k) \right\|_{\infty}\le 10^{-6}.
\end{equation*}
Recall that our algorithmic framework contains the quadratic penalty or Moreau-Yosida regularization technique \cite{Hintermueller2006,Ulbrich2011} as a special case (for $\BddMul^k:=0$). Since this method is rather popular (in particular for state-constrained optimal control problems), we have performed a numerical comparison to our augmented Lagrangian scheme. To make the comparison fair, we incorporated two modifications into the methods. For the Moreau-Yosida scheme, it does not make sense to update the penalty parameter conditionally, and it is therefore increased in every iteration. On the other hand, for the augmented Lagrangian method, recall that the penalty updating scheme \eqref{Eq:RhoTest} is only defined for $k\ge 1$. To enable a proper treatment of the penalty parameter in the first iteration, we use the updating scheme \eqref{Eq:RhoTest} with the right side replaced by $\tau\|\min\{-g(x^k),0\}\|_Z$ for $k=0$.

\subsection{The Obstacle Problem}\label{Sec:ObstProb}

We consider the well-known obstacle problem \cite{Ito1990,Rodrigues1987}. To this end, let $\Omega\subseteq\R^d$ be a bounded domain, and let $X:=Y:=H_0^1(\Omega)$, $Z:=L^2(\Omega)$. The obstacle problem considers the minimization problem
\begin{equation}\label{Eq:ObstProb}
    \min\ f(u) \quad\text{s.t.}\quad u\ge\psi,
\end{equation}
where $f(u):=\|\nabla u\|_{L^2(\Omega)}^2$ and $\psi\in X$ is a fixed obstacle. In order to formally describe this problem within our framework \eqref{Eq:Opt}, we make the obvious definition
\begin{equation*}
    g:X\to Y, \quad g(u):=\psi-u.
\end{equation*}
Using the Poincar\'e inequality, it is easy to see that $f$ is strongly convex on $X$ \cite[Thm.\ 6.30]{Adams2003}. Hence, the obstacle problem satisfies the requirements of Theorem \ref{Thm:Convex}, which implies that the augmented Lagrangian method is globally convergent. Furthermore, since $X=Y$, it follows that $g'(u)=-\operatorname{id}_X$ for every $u\in X$. Hence, the Zowe-Kurcyusz condition (cf.\ Definition \ref{Dfn:ZoweKurcyusz}) is trivially satisfied in every feasible point, which implies the boundedness of the dual iterates $(\lambda^k)$ by Theorem \ref{Thm:ZoweKurcyusz}.

In fact, the constraint function $g$ satisfies much more than the Zowe-Kurcyusz condition. For every $u\in X$, the mapping $g'(u)=-\operatorname{id}_X$ is bijective. Hence, if a subsequence $(x^k)_{\mathcal{K}}$ converges to a KKT point $\bar{x}$ of \eqref{Eq:ObstProb} and $\bar{\lambda}$ is the corresponding multiplier, then we obtain
\begin{equation*}
    f'(x^k)-\lambda^k=f'(x^k)+g'(x^k)^* \lambda^k \to_{\mathcal{K}} 0.
\end{equation*}
In other words, we see that $\lambda^k \to_{\mathcal{K}} f'(\bar{x})=\bar{\lambda}$, i.e.\ $(\lambda^k)_{\mathcal{K}}$ converges to the (unique) Lagrange multiplier corresponding to $\bar{x}$.

We now present some numerical results for $\Omega:=(0,1)^2$ and the obstacle
\begin{equation*}
    \psi(x,y):=\max\left\{0.1-0.5 \left\|
    \begin{pmatrix}
        x-0.5 \\ y-0.5
    \end{pmatrix}
    \right\|,0\right\},
\end{equation*}
cf.\ Figure \ref{Fig:Obst}. For the solution process, we choose $n\in\N$ and discretize $\Omega$ by means of a standard grid which consists of $n$ (interior) points per row or column, i.e.\ $n^2$ interior points in total. Furthermore, we use
\begin{equation*}
    f(u)=\|\nabla u\|_{L^2(\Omega)}^2 = -\dual{ \Delta u , u }_{X^*\times X}
    \quad\text{for all }u\in X
\end{equation*}
and approximate the Laplace operator by a standard five-point finite difference scheme. The subproblems occurring in Algorithm \ref{Alg:ALM} are unconstrained minimization problems which we solve by means of a standard semismooth Newton method.

\begin{table}[ht]\centering
\begin{tabular}{|r|ccc|ccc|}\hline
    & \multicolumn{3}{c|}{Augmented Lagrangian}
    & \multicolumn{3}{c|}{Moreau-Yosida} \\
    $n$ & outer & inner & final $\rho_k$ & outer & inner
    & final $\rho_k$ \\ \hline
    16 & 6 & 9 & $10^4$ & 7 & 11 & $10^7$ \\
    32 & 7 & 13 & $10^5$ & 7 & 15 & $10^7$ \\
    64 & 7 & 17 & $10^5$ & 7 & 18 & $10^7$ \\
    128 & 7 & 22 & $10^6$ & 8 & 22 & $10^8$ \\
    256 & 8 & 25 & $10^7$ & 8 & 27 & $10^8$ \\ \hline
\end{tabular}
\caption{Numerical results for the obstacle problem.}
\label{Tab:ObstProb}
\end{table}

\noindent
Table \ref{Tab:ObstProb} contains the inner and outer iteration numbers together with the final penalty parameter for different values of the discretization parameter $n$. Both the augmented Lagrangian and Moreau-Yosida regularization methods scale rather well with increasing dimension; in particular, the outer iteration numbers remain nearly constant. Performance-wise, the two methods perform very similarly, with the augmented Lagrangian method holding a slight advantage in terms of iteration numbers and penalty parameters.

\begin{figure}\centering
    \subfloat[Constraint function $\psi$]{\includegraphics[scale=0.5]{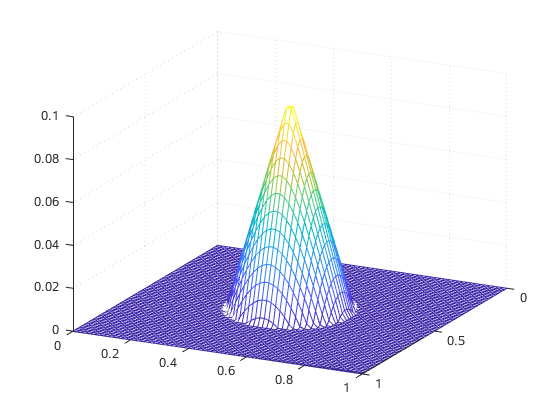}}
    \subfloat[Solution $\bar{u}$]{\includegraphics[scale=0.5]{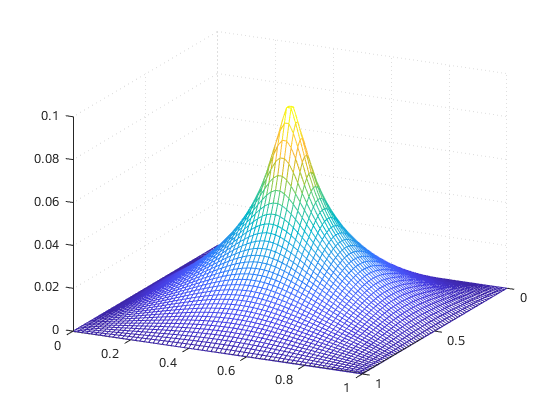}}
    \caption{Numerical results for the obstacle problem with $n=64$.}
    \label{Fig:Obst}
\end{figure}

\subsection{The Obstacle Bratu Problem} \label{Sec:BratuProb}

Let us briefly consider the obstacle Bratu problem \cite{Facchinei2003,Hoppe1989}, which we simply refer to as Bratu problem. This is a non-quadratic and nonconvex problem which differs from \eqref{Eq:ObstProb} in the choice of objective function. To this end, let
\begin{equation*}
    f(u):=\|\nabla u\|_{L^2(\Omega)}^2 - \alpha \int_{\Omega} e^{-u(x)} dx
\end{equation*}
for some fixed $\alpha>0$. To ensure well-definedness of $f$, we require $\Omega \subseteq \R^2$. As before, we set $X:=Y:=H_0^1(\Omega)$, $Z:=L^2(\Omega)$ and consider the minimization problem
\begin{equation}\label{Eq:BratuProb}
    \min\ f(u) \quad\text{s.t.}\quad u\ge\psi
\end{equation}
for some fixed obstacle $\psi\in X$; that is, $g(u):=\psi-u$. Before we proceed, let us first show that the function $f$ is well-defined and satisfies the assumptions \ref{Asm:Lsc} and \ref{Asm:C1}.

\begin{lem}\label{Lem:BratuProb}
    The function $f$ is well-defined, weakly lsc and continuously Fr\'echet differentiable from $H^1_0(\Omega)$ into $\R$.
\end{lem}
\begin{proof}
In the proof we will follow some arguments of \cite{ItoKunisch2002}.
It is only necessary to study the mapping $u\mapsto  \int_{\Omega} e^{u(x)} dx$.
The mapping $u\mapsto e^{u}$ maps bounded sets in $H^1_0(\Omega)$ to bounded sets in $L^p(\Omega)$ for all $p<\infty$, see \cite{PlumWieners2002}.
Let $(u_n)$ be a sequence converging weakly to $u$ in $H^1_0(\Omega)$. By compact embeddings, this sequence
converges strongly to $u$ in $L^p(\Omega)$ for all $p<\infty$. After extracting a subsequence if necessary, we have $u_n\to u$ pointwise
almost everywhere and $e^{u_n}\rightharpoonup \tilde e$ in $L^p(\Omega)$ for all $p<\infty$.
By a result of Brezis \cite[Lemma 3, page 126]{Brezis1971}, it follows $\tilde e = e^{u}$. This proves the weak lower semicontinuity of $f$.

Let $u,h \in H^1_0(\Omega)$ be given. Let us write
\[
 e^{u+h}-e^u-e^uh = \int_0^1 \int_0^1 e^{u+sth}sh^2 \ dt \ ds,
\]
which implies
\[
 \| e^{u+h}-e^u-e^uh\|_{L^1(\Omega)} \le \frac12\|e^{u+|h|}\|_{L^2(\Omega)} \|h\|_{L^4(\Omega)}^2
 \le c\|e^{u+|h|}\|_{L^2(\Omega)} \|h\|_{H^1(\Omega)}^2.
\]
Using the boundedness property of the mapping $u\mapsto e^{u}$ mentioned above, it follows
$ \| e^{u+h}-e^u-e^uh\|_{L^1(\Omega)} =o(\|h\|_{H^1(\Omega)})$, which implies the Fr\'echet differentiability of $f$.
\end{proof}

\noindent
Due to the constraint $u\ge \psi$, the functional $f$ is bounded from below on the feasible set of \eqref{Eq:BratuProb}.
Together with the lower-semicontinuity result this implies the existence of solutions of the minimization problem \eqref{Eq:BratuProb}.
Note that this statement is no longer valid if $\Omega\subseteq\R^d$ with $d\ge3$.

From a theoretical point of view, the Bratu problem is much more difficult 
than the obstacle problem from Section \ref{Sec:ObstProb}. While the 
constraint function is equally well-behaved, the objective function in 
\eqref{Eq:BratuProb} is neither quadratic nor convex. Hence, we cannot 
apply Theorem~\ref{Thm:Convex} or the theory from Section~\ref{Sec:Global},
wheras the KKT-like convergence results from Sections~\ref{Sec:KKT}
and \ref{Sec:Convergence} still hold.

To analyse how our method behaves in practice, we again considered 
$\Omega:=(0,1)^2$ and implemented the Bratu problem using the same obstacle 
and a similar implementation as we did for the standard obstacle problem. 
The resulting images are given in Figure \ref{Fig:Bratu}, and some iteration
numbers are given in Table \ref{Tab:Bratu}.
\begin{table}[ht]\centering
\begin{tabular}{|r|ccc|ccc|}\hline
    & \multicolumn{3}{c|}{Augmented Lagrangian}
    & \multicolumn{3}{c|}{Moreau-Yosida} \\
    $n$ & outer & inner & final $\rho_k$ & outer & inner
    & final $\rho_k$ \\ \hline
    16 & 6 & 13 & $10^4$ & 7 & 15 & $10^7$ \\
    32 & 7 & 17 & $10^5$ & 7 & 17 & $10^7$ \\
    64 & 7 & 19 & $10^5$ & 7 & 19 & $10^7$ \\
    128 & 8 & 24 & $10^6$ & 8 & 23 & $10^8$ \\
    256 & 8 & 24 & $10^6$ & 8 & 28 & $10^8$ \\ \hline
\end{tabular}
\caption{Numerical results for the Bratu problem.}
\label{Tab:Bratu}
\end{table}
As with the obstacle problem, we observe that both the augmented Lagrangian and Moreau-Yosida regularization methods scale well with increasing dimension, and the augmented Lagrangian method once again holds a certain advantage in terms of iteration numbers and penalty parameters. In fact, the gap between the two methods is slightly bigger than for the standard obstacle problem.

\begin{figure}\centering
    \subfloat[Constraint function $\psi$]{\includegraphics[scale=0.5]{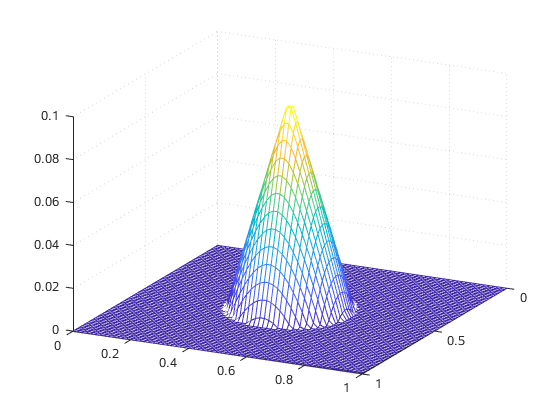}}
    \subfloat[Solution $\bar{u}$]{\includegraphics[scale=0.5]{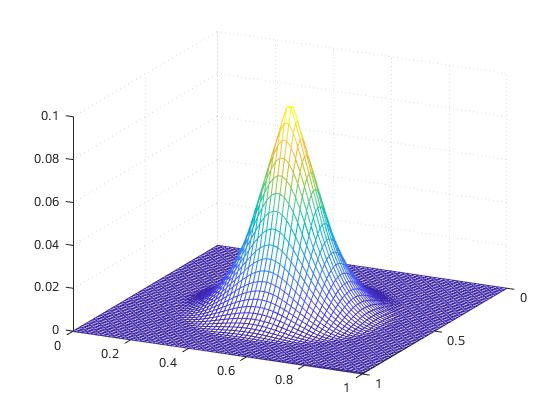}}
    \caption{Numerical results for the Bratu problem with $\alpha=1$ and $n=64$.}
    \label{Fig:Bratu}
\end{figure}

\subsection{Optimal Control Problems}

We now turn to a class of optimal control problems subject to a semilinear elliptic equation. Let $\Omega\subseteq \R^d$, $d=2,3$, be a bounded Lipschitz domain.
The control problem we consider consists of minimizing the functional
\begin{equation}\label{Eq:OptCont}
    J(y,u):=\frac12\|y-y_d\|_{L^2(\Omega)}^2 + \frac \alpha2\|u\|_{L^2(\Omega)}^2
\end{equation}
subject to $y\in H^1_0(\Omega)\cap C(\bar\Omega)$ and $u\in L^2(\Omega)$ satisfying the
semilinear equation
\begin{equation*}
    -\Delta y + d(y) = u \quad \text{in } H^1_0(\Omega)^*
\end{equation*}
and the pointwise state constraints
\begin{equation}\label{Eq:OptContC}
    y\ge y_c \quad\text{in } \Omega.
\end{equation}
Here, $\alpha$ is a positive parameter, $y_d\in L^2(\Omega)$, and $y_c\in C(\bar\Omega)$ with $y_c\le 0$ on $\partial\Omega$ are given functions. The nonlinearity $d$ in the elliptic equation is induced by a function $d:\R\to \R$, which is assumed to be continuously differentiable and monotonically increasing.

Before we can apply the augmented Lagrangian method to \eqref{Eq:OptCont}, we need to formally eliminate the state equation coupling the variables $y$ and $u$. Due to elliptic regularity results, this equation admits for each control $u\in L^2(\Omega)$ a uniquely determined weak solution $y\in H^1_0(\Omega)\cap C(\bar\Omega)$. Moreover, the mapping $u\mapsto y$ is Fr\'echet differentiable in this setting \cite[Thm.\ 4.17]{Troeltzsch2010}. Let us denote this mapping by $S$. Using $S$, we can eliminate the state equation to obtain an optimization problem with inequality constraints:
\begin{equation}\label{Eq:OptCont2}
    \min\ J(S(u),u) \quad \text{s.t.}\quad S(u)\ge y_c.
\end{equation}
We can now apply Algorithm \ref{Alg:ALM} to this problem. The inequality $S(u)\ge y_c$ has to be understood in the sense of $C(\bar\Omega)$, which necessitates the choice $Y:=C(\bar\Omega)$. Furthermore, we have $X:=Z:=L^2(\Omega)$. Assuming a linearized Slater condition, one can prove that the Zowe-Kurcyusz condition is fulfilled, and there exists a Lagrange multiplier $\lambda \in C(\bar\Omega)^*$ to the inequality constraint $S(u)\ge y_c$, see, e.g., \cite[Thm.\ 6.8]{Troeltzsch2010}.

The subproblems generated by Algorithm \ref{Alg:ALM} are unconstrained optimization problems. By reintroducing the state variable $y$, we can write these subproblems as
\begin{equation}\label{Eq:OptContS}
    \min\ J(y,u)+\frac{\rho_k}{2}\left\| \left( y_c-y+\frac{\BddMul^k}{\rho_k} \right)_+ \right\|^2
    \quad\text{s.t.}\quad y=S(u).
\end{equation}
Hence, we have transformed \eqref{Eq:OptCont} into a sequence of optimal control problems which include the state equation but not the pointwise constraint \eqref{Eq:OptContC}.

Let us proceed with some numerical results. As a test problem, we chose an example similar to the one presented in \cite{Neitzel2015}, where $\Omega:=(0,1)^2$, $d(y):=y^3$, $\alpha:=10^{-3}$, and
\begin{equation*}
    y_c(x):=-\frac{2}{3}+\frac{1}{2}
    \min\{ x_1+x_2,1+x_1-x_2,
    1-x_1+x_2,2-x_1-x_2\}.
\end{equation*}
Clearly, in this setting, \eqref{Eq:OptCont} and its reformulation \eqref{Eq:OptCont2} are nonconvex problems. We solve the subproblems \eqref{Eq:OptContS} with the MATLAB\textsuperscript{\textregistered} function \texttt{fmincon}, where the Hessian of the objective is replaced by a semismooth version thereof.
\begin{table}\centering
\begin{tabular}{|r|ccc|ccc|}\hline
    & \multicolumn{3}{c|}{Augmented Lagrangian}
    & \multicolumn{3}{c|}{Moreau-Yosida} \\
    $n$ & outer & inner & final $\rho_k$ & outer & inner
    & final $\rho_k$ \\ \hline
    16 & 6 & 16 & $10^4$ & 6 & 19 & $10^6$ \\
    32 & 7 & 21 & $10^5$ & 7 & 22 & $10^7$ \\
    64 & 7 & 23 & $10^6$ & 7 & 25 & $10^7$ \\
    128 & 7 & 26 & $10^6$ & 8 & 30 & $10^8$ \\
    256 & 8 & 31 & $10^7$ & 9 & 37 & $10^9$ \\ \hline
\end{tabular}
\caption{Numerical results for the optimal control problem.}
\label{Tab:OptCont}
\end{table}
Table \ref{Tab:OptCont} contains the resulting iteration numbers and final penalty parameters for both the augmented Lagrangian and Moreau-Yosida regularization methods. As with the previous examples, both methods scale well with increasing dimension, and the augmented Lagrangian method is more efficient in terms of iteration numbers and penalty parameters.

\begin{figure}[ht]\centering
    \subfloat[Constraint function $y_c$]{\includegraphics[scale=0.5]{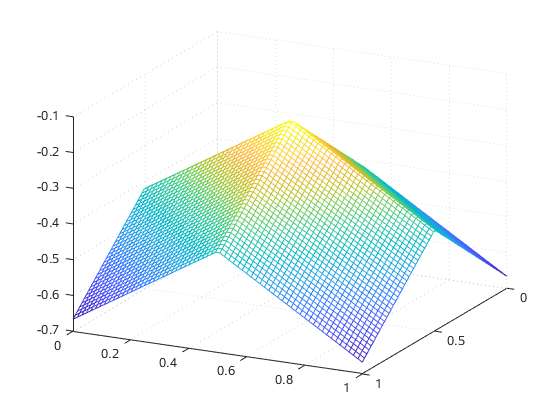}}
    \subfloat[Optimal state $\bar{y}$]{\includegraphics[scale=0.5]{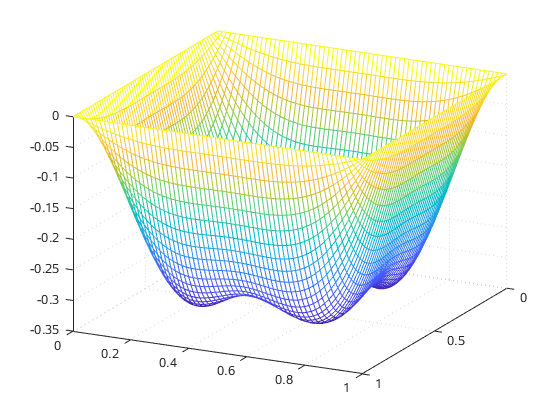}}\\
    \subfloat[Optimal control $\bar{u}$]{\includegraphics[scale=0.5]{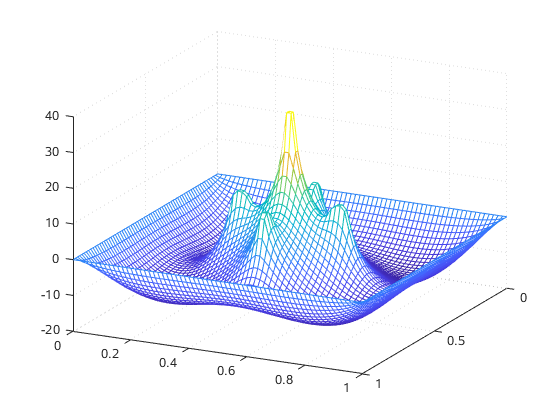}}
    \subfloat[Lagrange multiplier $\bar{\lambda}$]{\includegraphics[scale=0.5]{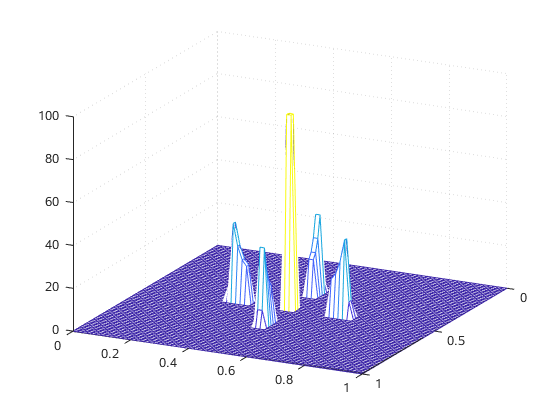}}
    \caption{Numerical results for the optimal control problem with $n=64$.}
    \label{Fig:OptCont}
\end{figure}

\noindent
The state constraint $y_c$ and the results of our method are given in Figure \ref{Fig:OptCont}. It is interesting to note that the multiplier $\bar{\lambda}$ appears to be much less regular than the optimal control $\bar{u}$ and state $\bar{y}$. This is not surprising because, due to our construction, we have
\begin{equation*}
    \bar{u}\in L^2(\Omega),\quad \bar{y}\in C(\bar{\Omega}),
    \quad\text{and}\quad \bar{\lambda}\in C(\bar{\Omega})^*.
\end{equation*}
The latter is well-known to be the space of Radon measures on $\bar{\Omega}$, which is a superset of $L^2(\Omega)$. In fact, the convergence data shows that the (discrete) $L^2$-norm of $\bar{\lambda}$ grows approximately linearly as $n$ increases, possibly even diverging to $+\infty$, which suggests that the underlying (infinite-dimensional) problem \eqref{Eq:OptCont} does not admit a multiplier in $L^2(\Omega)$ but only in $C(\bar{\Omega})^*$.

\section{Final Remarks}

We have presented an augmented Lagrangian method for the solution of optimization problems in Banach spaces, which is essentially a generalization of the modified augmented Lagrangian method from \cite{Birgin2014}. Furthermore, we have shown how the method can be applied to well-known problem classes, and the corresponding numerical results appear quite promising. In particular, the method appears to be (slightly) more efficient than the well-known Moreau-Yosida regularization scheme, especially with regard to the behavior of the penalty parameter.

From a theoretical point of view, the main strength of our method is the ability to deal with very general classes of inequality constraints; in particular, inequality constraints with infinite-dimensional image space. Other notable features include desirable convergence properties for nonsmooth problems, the ability to find KKT points of arbitrary nonlinear (and nonconvex) problems, and a global convergence result which covers many prominent classes of convex problems. We believe the sum of these aspects to be a substantial contribution to the theory of augmented Lagrangian methods.

Another key concern in our work is the compatibility of the algorithm with suitable constraint qualifications. To deal with this matter properly, we investigated the well-known Zowe-Kurcyusz regularity condition \cite{Zowe1979}, see also Robinson \cite{Robinson1976}, and showed that this condition can be used to guarantee the boundedness of suitable multiplier sequences corresponding to asymptotic KKT conditions. While the main application of this result is clearly the boundedness of the multiplier sequence generated by the augmented Lagrangian method, we state explicitly that the underlying theory is independent of our specific algorithm. With the understanding that most iterative methods for constrained optimization usually satisfy the KKT conditions in an asymptotic sense, we hope that this aspect of our theory will facilitate similar research into other methods or find applications in other topics.

\section*{Acknowledgement}

The authors would like to thank two anonymous referees for their helpful suggestions in improving the paper.

\bibliographystyle{abbrv}
\bibliography{OptInfALM2}

\end{document}